\documentclass[a4paper,11pt]{article}    

\usepackage{graphicx}   
  
\usepackage{latexsym}
\usepackage{amsmath}
\usepackage{amssymb}
\usepackage{amsthm}
\usepackage{psfrag}
\usepackage{enumerate}
\usepackage{color}

\title{
Note on Minimization of 
Quasi \Mnat-convex Functions%
\footnote{
This work was supported by 
JSPS KAKENHI Grant Numbers JP23K11001 and JP23K10995.
}
}

\author{Kazuo Murota%
\footnote{
The Institute of Statistical Mathematics,
Tokyo 190-8562, Japan; 
and
Faculty of Economics and Business Administration,
Tokyo Metropolitan University, 
Tokyo 192-0397, Japan,
\texttt{murota@tmu.ac.jp}
}
 \and
Akiyoshi Shioura%
\footnote{
Department of Industrial Engineering and Economics,
Tokyo Institute of Technology,
Tokyo 152-8550, Japan
  \texttt{shioura.a.aa@m.titech.ac.jp}}
}

\date{May 29, 2023 / June 18, 2023 / September 5, 2023 /November 27, 2023}

\newcommand{\suppp}{{\rm supp}\sp{+}}
\newcommand{\suppm}{{\rm supp}\sp{-}}

\newcommand{\RR}{\mathbb{R}}

\newcommand{\Z}{\mathbb{Z}}

\newcommand{\dom}{{\rm dom\,}}

\newcommand{\todaye}{\the\year/\the\month/\the\day}

\newcommand{\Rinf}{\RR \cup \{ +\infty \}}

\newcommand{\Mnat}{{M$^{\natural}$}}

\newcommand{\SSQM}{\mbox{\rm (SSQM)}}
\newcommand{\SSQMb}{\mbox{\rm\bf (SSQM)}}
\newcommand{\SSQMn}{\mbox{\rm (SSQM$\sp{\natural}$)}}
\newcommand{\SSQMnb}{\mbox{\rm\bf (SSQM$\sp{\natural}$)}}

\newcommand{\Nzero}{N\cup\{0\}}

\newcommand{\norm}[1]{\lVert#1\rVert}

\newtheorem{theorem}{Theorem}[section]
\newtheorem{lemma}[theorem]{Lemma}
\newtheorem{corollary}[theorem]{Corollary}

\newtheorem{example}[theorem]{Example}

\numberwithin{equation}{section}

\begin{document}

\maketitle

\begin{abstract}
For a class of discrete quasi convex functions
called semi-strictly quasi \Mnat-convex functions,
we investigate fundamental issues
relating to minimization,
such as
optimality condition by local optimality,
minimizer cut property,
geodesic property,
and proximity property.
Emphasis is put on comparisons with (usual) \Mnat-convex functions.
The same optimality condition
and a weaker form of the minimizer cut property
hold for semi-strictly quasi \Mnat-convex functions,
while geodesic property and proximity property fail.
\end{abstract}

\newpage

\section{Introduction}
\label{sec:1}

 In this paper, we deal with a class of discrete quasi convex functions
called semi-strictly quasi \Mnat-convex functions
\cite{FS05sbst} (see also \cite{MS03quasi}).
 The concept of semi-strictly quasi \Mnat-convex function
is introduced as a ``quasi convex'' version of \Mnat-convex function
\cite{MS99gp},
which is a major concept in the theory of discrete convex analysis 
introduced as a variant of M-convex function 
 \cite{Mstein96,Mdca98,Mdcasiam}.
 Application of (semi-strictly) quasi \Mnat-convex functions
can be found in mathematical economics \cite{FS05sbst,MY15mor} 
and
operations research \cite{CM21mnat}.

 An M-convex function is defined as a function $f:\Z^n \to \Rinf$
satisfying a certain exchange axiom (see Section \ref{sec:quasiM}), 
which implies that
the effective domain $\dom f = \{x \in \Z^n \mid f(x) < + \infty\}$
is contained in a hyperplane of the form $\sum_{i=1}^n x(i)= r$
for some $r \in \Z$.
 Due to this fact, it is natural to consider 
the projection of an M-convex function
to the $(n-1)$-dimensional space along a coordinate axis, 
which is called an \Mnat-convex function.
 A nontrivial argument shows that an \Mnat-convex  function $f:\Z^n \to \Rinf$
is characterized by the following exchange axiom:
\begin{quote}
{\bf (\Mnat-EXC)}
 $\forall x, y \in \dom f$, $\forall i \in \suppp (x - y)$,
$\exists j \in \suppm(x - y)\cup\{0\}$ such that
\begin{equation}
\label{eqn:def-Mnat-ineq}
  f(x) + f(y) \geq f(x - \chi_{i} + \chi_{j}) 
 + f(y + \chi_{i} - \chi_{j}),
\end{equation}
\end{quote}
where $N = \{1,2,\ldots, n\}$, 
$\chi_i \in \{0, 1\}^n$  is the characteristic vector of $i \in N$,
$\chi_0 = 0$, and
\begin{align*}
& \suppp(x-y) = \{i \in N \mid x(i) > y(i)\},  \qquad
\suppm(x-y) = \{j \in N \mid x(j) < y(j)\}.
\end{align*}

 The inequality \eqref{eqn:def-Mnat-ineq} implies that
at least one of the following three conditions holds:
\begin{align}
& f(x - \chi_i + \chi_j)< f(x),
\label{eqn:def-SSQMnat-ineq-1}\\
& f(y + \chi_i - \chi_j) < f(y),
\label{eqn:def-SSQMnat-ineq-2}\\
& f(x - \chi_i + \chi_j) = f(x) \mbox{ and } f(y + \chi_i - \chi_j) = f(y).
\label{eqn:def-SSQMnat-ineq-3}
\end{align}
 Using this, a semi-strictly quasi \Mnat-convex function 
(s.s.\;quasi \Mnat-convex function, for short)
is defined as follows:
 $f:\Z^n \to \Rinf$ 
is called an s.s.\;quasi \Mnat-convex function
if it satisfies the following exchange axiom:
\begin{quote}
{\bf \SSQMnb}
 $\forall x, y \in \dom f$, $\forall i \in \suppp (x - y)$,
$\exists j \in \suppm(x - y)\cup\{0\}$ satisfying
at least one of the conditions
\eqref{eqn:def-SSQMnat-ineq-1},
\eqref{eqn:def-SSQMnat-ineq-2}, and \eqref{eqn:def-SSQMnat-ineq-3}.
\end{quote}

 The main aim of this paper is to investigate fundamental issues 
relating to minimization of an s.s.\;quasi \Mnat-convex function.
 It is known that minimizers of an M-convex function
have various nice properties (to be described in 
Section \ref{sec:prop-M-conv} of Appendix)
such as
\begin{quote}
 $\bullet$ optimality condition by local optimality, 
\\
 $\bullet$ minimizer cut property,
\\
 $\bullet$ geodesic property,
\\
 $\bullet$ proximity property.
\end{quote}
The definition of \Mnat-convex function implies that
these properties of M-convex functions
are inherited by \Mnat-convex functions,
as shown in Section \ref{sec:prop}.
 In this paper, we examine which of the above properties are satisfied by
s.s.\;quasi \Mnat-convex functions.
 For each of the properties, 
if it holds for s.s.\;quasi \Mnat-convex functions,
we describe the precise statement of the property in question
and give its proof;
otherwise, we provide an example to show the failure of the property.

It is added that
there is a notion called ``s.s.\;quasi M-convex function''
\cite{Mdcasiam,MS03quasi},
which corresponds directly to M-convexity.
Although M-convex and \Mnat-convex functions are known to be essentially
equivalent,
it turns out that
their quasi-convex versions, namely,
s.s.\;quasi M-convexity and s.s.\;quasi \Mnat-convexity,
are significantly different.
We also discuss such subtle points in Section \ref{sec:quasiM}.


\section{Properties on Minimization of Quasi \Mnat-convex Functions}
\label{sec:prop}

\subsection{Optimality Condition by Local Optimality}
\label{sec:loc-opt}

In this section we consider an optimality condition
for minimization in terms of local optimality
and also a minimization algorithm based on the optimality condition.
Before dealing with quasi \Mnat-convex functions,
we describe the existing results for \Mnat-convex functions.

A minimizer $x^*$ of an \Mnat-convex function $f$ can be characterized by 
the local minimality within the neighborhood
consisting of vectors $y \in \Z^n$ with  $\|y - x^*\|_1 \le 2$.

\begin{theorem}[{cf.~{\cite[Theorem 2.4]{Mstein96}, \cite[Theorem 2.2]{Msbmfl99}}}]
\label{thm:Mminimizer-Mnat}
 Let $f: \Z^n \to \Rinf$ be an \Mnat-convex function.
 A vector $x^* \in \dom f$ 
is a minimizer of $f$ if and only if
\begin{align}
\label{eqn:loc-min-global-min2}
f(x^* - \chi_i + \chi_j) \ge f(x^*) \qquad (i, j \in \Nzero).
\end{align}
\end{theorem}

 Theorem \ref{thm:Mminimizer-Mnat} makes it possible to
apply the following steepest descent algorithm to find
a minimizer of an \Mnat-convex function.

\begin{flushleft}
 \textbf{Algorithm} {\sc BasicSteepestDescent} 
\\
 \textbf{Step 0:} 
 Let $x_0 \in \dom f$ be an arbitrarily chosen initial vector.
 Set $x:=x_0$.
\\
 \textbf{Step 1:} 
If 
$f(x - \chi_{i} + \chi_{j}) \ge f(x)$
for every $i, j \in \Nzero$, 
then output $x$ and stop.\\
 \textbf{Step 2:} 
 Find $i, j \in \Nzero$ that minimize $f(x - \chi_{i} + \chi_{j})$.\\
 \textbf{Step 3:} 
 Set $x := x - \chi_{i} + \chi_{j}$ and go to Step 1.
\end{flushleft}

\begin{corollary}[{cf.~{\cite{Shi98min}}}]
 For an \Mnat-convex function $f: \Z^n \to \Rinf$ with 
{\nobreak $\arg\min f\ne\emptyset$},
Algorithm {\sc BasicSteepestDescent}  
finds a minimizer of $f$ in a finite number of iterations.
\end{corollary}

The optimality condition for \Mnat-convex functions 
in Theorem \ref{thm:Mminimizer-Mnat}
can be generalized to s.s.\;quasi \Mnat-convex functions.

\begin{theorem}
\label{thm:Mminimizer-SSQMn}
 Let $f: \Z^n \to \Rinf$ be a function satisfying {\SSQMn}. 
 A vector $x^* \in \dom f$ 
is a minimizer of $f$ if and only if
the condition \eqref{eqn:loc-min-global-min2} holds.
\end{theorem}

\noindent
 The ``only if'' part of the theorem is easy to see. 
 The ``if'' part is implied immediately by the following lemma.

\begin{lemma}
  Let $f: \Z^n \to \Rinf$ be a function satisfying {\SSQMn}. 
 For $x, y \in \dom f$, if $f(x) > f(y)$, then
there exist some $i\in \suppp(x-y)\cup\{0\}$ and $j\in \suppm(x-y)\cup\{0\}$ 
satisfying 
$f(x)>f(x-\chi_i + \chi_j)$.
\end{lemma}

\begin{proof}
 Putting
\[
 \alpha = \sum_{i\in \suppp(x-y)}|x(i)-y(i)|,
\qquad
\beta = \sum_{j\in \suppm(x-y)}|x(j)-y(j)|,
\]
we prove the lemma by induction on the pair of values  $(\alpha, \beta)$.
 If $\alpha \le 1$ and $\beta \le 1$, then
$y = x - \chi_i + \chi_j$ for some $i,j \in \Nzero$, 
and therefore
the claim holds immediately.

 Suppose $\alpha \ge 2$ and let $i \in \suppp(x-y)$.
 By {\SSQMn} applied to $x$, $y$, and $i$,
there exists some  $j \in \suppm(x-y)\cup\{0\}$ satisfying
$f(x - \chi_i + \chi_j) < f(x)$ or
$f(y + \chi_i - \chi_j) \le f(y)$ (or both).
 In the former case, we are done.
 In the latter case,
we can apply the induction hypothesis
to $x$ and $y'=y + \chi_i - \chi_j$
to obtain some 
$i'\in \suppp(x-y')\cup\{0\} \subseteq \suppp(x-y)\cup\{0\}$ 
and $j'\in \suppm(x-y')\cup\{0\}\subseteq \suppm(x-y)\cup\{0\}$ 
satisfying 
$f(x)>f(x-\chi_{i'} + \chi_{j'})$.
 The proof for the case $\beta \ge 2$ is similar.
\end{proof}

 It follows from Theorem \ref{thm:Mminimizer-SSQMn} that
we can also apply the steepest descent algorithm to find
a minimizer of an s.s.\;quasi \Mnat-convex function.

\begin{corollary}
 For a function $f: \Z^n \to \Rinf$ with $\arg\min f \ne \emptyset$
satisfying {\SSQMn},
Algorithm {\sc BasicSteepestDescent} 
finds a minimizer of $f$ in a finite number of iterations.
\end{corollary}


\subsection{Minimizer Cut Property}
\label{sec:min-cut}

 The minimizer cut property, originally shown for M-convex functions
\cite[Theorem 2.2]{Shi98min}
(see also Theorem \ref{thm:min-cut-M} in Appendix),
states that 
a separating hyperplane between a given vector $x$ and some minimizer
can be found 
by using the steepest descent direction at $x$
(i.e., vector $ - \chi_{i} + \chi_{j}$ with $i, j \in N$
that minimizes $f(x - \chi_{i} + \chi_{j})$).
 By rewriting  the minimizer cut property
for M-convex functions 
based on the relationship between M-convexity and \Mnat-convexity,
we obtain the following minimizer cut property for \Mnat-convex functions,
where $y(N)= \sum_{i \in N} y(i)$ for $y \in \Z^n$.

\begin{theorem}[{cf.~{\cite[Theorem 2.2]{Shi98min}}}]
\label{thm:min-cut-Mnat}
 Let $f: \Z^n \to \Rinf$ be an \Mnat-convex function with $\arg\min f \ne \emptyset$,
and $x \in \dom f$ be a vector with $x \not\in \arg\min f$.
 For a pair $(i,j)$ of distinct elements in $\Nzero$
minimizing the value $f(x - \chi_i + \chi_j)$,
there exists some minimizer $x^*$ of~$f$ satisfying
\begin{equation}
\begin{cases}
x^*(i) \le x(i)-1, \  
x^*(j) \ge x(j)+1 
& (\mbox{if }i,j \in N),\\
x^*(i) \le x(i)-1, \  
x^*(N) \le x(N) -1
& (\mbox{if }i \in N,\ j =0),\\
x^*(j) \ge x(j)+1, \  
x^*(N) \ge x(N) +1
& (\mbox{if }i =0,\ j \in N).
 \end{cases}
\label{eqn:mincut-cond-v-Mnat}
\end{equation}
\end{theorem}
\noindent
 Other variants of the 
minimizer cut property of \Mnat-convex functions
are given in Section \ref{sec:min-cut-remark},
which capture the technical core of Theorem \ref{thm:min-cut-Mnat}.

 Using Theorem \ref{thm:min-cut-Mnat} we can provide an upper bound
for the number of iterations in the following variant
of the steepest descent algorithm, where
$\dom f$ is assumed to be bounded, and 
the integer interval $[\ell, u] = \{x \in \Z^n \mid \ell \le x \le u\}$
always contains
a minimizer of $f$.

\begin{flushleft}
 \textbf{Algorithm} {\sc ModifiedSteepestDescent} 
\\
 \textbf{Step 0:} 
 Let $x_0 \in \dom f$ be an arbitrarily chosen initial vector.
 Set $x:=x_0$.\\
 \phantom{\textbf{Step 0:}} 
 Let $\ell, u \in \Z^n$ be vectors such that
$\dom  f \subseteq [\ell, u]$.
\\
 \textbf{Step 1:} 
If 
$f(x - \chi_{i} + \chi_{j}) \ge f(x)$
for every $i, j \in \Nzero$
with $x - \chi_{i} + \chi_{j} \in [\ell, u]$, \\
 \phantom{\textbf{Step 0:}} 
then output $x$ and stop.\\
 \textbf{Step 2:} 
 Find $i, j \in \Nzero$ 
with  $x - \chi_{i} + \chi_{j} \in [\ell, u]$ 
that minimize $f(x - \chi_{i} + \chi_{j})$.
\\
 \textbf{Step 3:} 
 Set $x := x - \chi_{i} + \chi_{j}$,
$u(i):=x(i)- 1$ if $i \in N$, 
and $\ell(j):=x(j)+1$ if $j \in N$.\\
 \phantom{\textbf{Step 0:}} 
 Go to Step 1.
\end{flushleft}

\noindent
 We define the \textit{L$_\infty$-diameter}
of a bounded set $S \subseteq \Z^n$ by
\[
 L_\infty(S) = \max\{\|x-y\|_\infty \mid x, y \in S\}.
\]
 
\begin{corollary}[{cf.~{\cite[Section 2]{Shi98min}}}]
 For  an \Mnat-convex function $f: \Z^n \to \Rinf$ with 
a bounded effective domain,
Algorithm {\sc ModifiedSteepestDescent} 
finds a minimizer of $f$ in $O(nL)$ iterations with $L =L_\infty(\dom f)$.
\end{corollary}

 While the number of iterations in
the algorithm {\sc ModifiedSteepestDescent} is proportional to
the L$_\infty$-diameter of $\dom f$, 
the domain reduction approach in \cite{Shi98min}
(see Section \ref{sec:domain_reduction};
see also \cite[Section 10.1.3]{Mdcasiam}),
combined with the minimizer cut property, makes it possible to
speed up the computation of a minimizer.

\begin{corollary}[{cf.~{\cite[Theorem 3.2]{Shi98min}}}]
\label{coro:domain-reduction-Mnat}
 Let $f: \Z^n \to \Rinf$ be an \Mnat-convex function with 
a bounded effective domain,
and suppose that a function evaluation oracle for $f$
and a vector in $\dom f$ are available.
 Then, a minimizer of $f$ can be obtained in $O(n^4 (\log L)^2)$ time
with $L =L_\infty(\dom f)$.
\end{corollary}
\noindent
 Note that faster polynomial-time algorithms
based on the scaling technique
are also available for \Mnat-convex function minimization
\cite{Shimin04,Tam05scale}.

 An s.s.\;quasi \Mnat-convex function
satisfies the following weaker
statement than Theorem \ref{thm:min-cut-Mnat}.
To be specific, the inequality
$x^*(N) \le x(N) -1$ in the second case ($i \in N,\; j =0$)
of \eqref{eqn:mincut-cond-v-Mnat} 
is missing in \eqref{eqn:mincut-cond-v} below,
and
the inequality
$x^*(N) \ge x(N) +1$ in the third case ($i =0,\; j \in N$)
of \eqref{eqn:mincut-cond-v-Mnat}
is missing in \eqref{eqn:mincut-cond-v};
an example illustrating this difference is given later
in Example \ref{ex:min-cut-qMnat-1}.

\begin{theorem}
\label{thm:min-cut-qMnat}
 Let $f: \Z^n \to \Rinf$ be a function 
with {\SSQMn} satisfying $\arg\min f \ne \emptyset$,
and $x \in \dom f$ be a vector with $x \not\in \arg\min f$.
 For a pair $(i,j)$ of distinct elements in $\Nzero$
minimizing the value $f(x - \chi_i + \chi_j)$,
there exists some minimizer $x^*$ of~$f$ satisfying
\begin{equation}
\left\{\begin{array}{ll}
 x^*(i) \le x(i)-1, \  
 x^*(j) \ge x(j)+1 
 & (\mbox{if }i,j \in N),\\
 x^*(i) \le x(i)-1
 & (\mbox{if }i \in N,\ j =0),\\
 x^*(j) \ge x(j)+1
 & (\mbox{if }i =0,\ j \in N).
\end{array} 
\right.
\label{eqn:mincut-cond-v}
\end{equation}
\end{theorem}
\noindent
 A proof of this theorem is given in Section \ref{sec:proof-min-cut}.

 Using Theorem \ref{thm:min-cut-qMnat} we can  obtain the same upper bound
for the number of iterations in 
the algorithm {\sc ModifiedSteepestDescent} applied to
s.s.\;quasi \Mnat-convex functions.

\begin{corollary}
 For a function $f: \Z^n \to \Rinf$ with a bounded effective domain
satisfying {\SSQMn},
Algorithm {\sc ModifiedSteepestDescent} 
finds a minimizer of $f$ in $O(nL)$ iterations with $L =L_\infty(\dom f)$.
\end{corollary}

A combination of Theorem \ref{thm:min-cut-qMnat}
with the domain reduction approach
(described in Section \ref{sec:domain_reduction})
makes it possible to find a minimizer of an
s.s.\;quasi \Mnat-convex function in time polynomial in
$n$ and $\log L_\infty(\dom f)$,
provided that $\dom f$ is an \Mnat-convex set.
 Note that the effective domain of an s.s.\;quasi \Mnat-convex function
is not necessarily an \Mnat-convex set (see \eqref{eqn:def-Mnat-conv-set} in
Section \ref{sec:domain_reduction}
for the definition of \Mnat-convex set).

\begin{corollary}
\label{cor:domain-reduction-quasiMnat}
 Let $f: \Z^n \to \Rinf$ be a function satisfying {\SSQMn},
and suppose that the effective domain of $f$ is bounded and \Mnat-convex.
 Also, suppose that
a function evaluation oracle for $f$ and a vector in $\dom f$ are available.
 Then, a minimizer of $f$ can be obtained in $O(n^4 (\log L)^2)$ time
with $L =L_\infty(\dom f)$.
\end{corollary}
\noindent
 A proof of this corollary is given in Section \ref{sec:domain_reduction}.

\begin{example}\rm 
\label{ex:min-cut-qMnat-1}
 This example shows that
the statement of Theorem~\ref{thm:min-cut-Mnat},
stronger than Theorem \ref{thm:min-cut-qMnat},
is not true for s.s.\;quasi \Mnat-convex functions.
 Consider the function $f: \Z^3 \to \Rinf$ defined by
\begin{align*}
& f(2,1,0)=f(2,0,1) = 0, && f(1,1,0)=f(1,0,1) = 1,\\ 
& f(0,1,1)=f(0,0,2) = 2, &
& f(1,1,1)=f(1,0,2) = 3,\\ 
& f(0,1,2)=4, &
& f(x_1, x_2, x_3) = + \infty \mbox{ otherwise}
\end{align*}
(see Figure \ref{fig:geodesic-example2}).
 This function $f$ satisfies {\SSQMn}
and has two minimizers  $y^* = (2,1,0)$  and $y^{**}=(2,0,1)$
(denoted by {\boldmath ${\bf \bigcirc}$} in Figure \ref{fig:geodesic-example2}).
\begin{figure}[t]
 \begin{center}
 \includegraphics[width=0.4\textwidth]{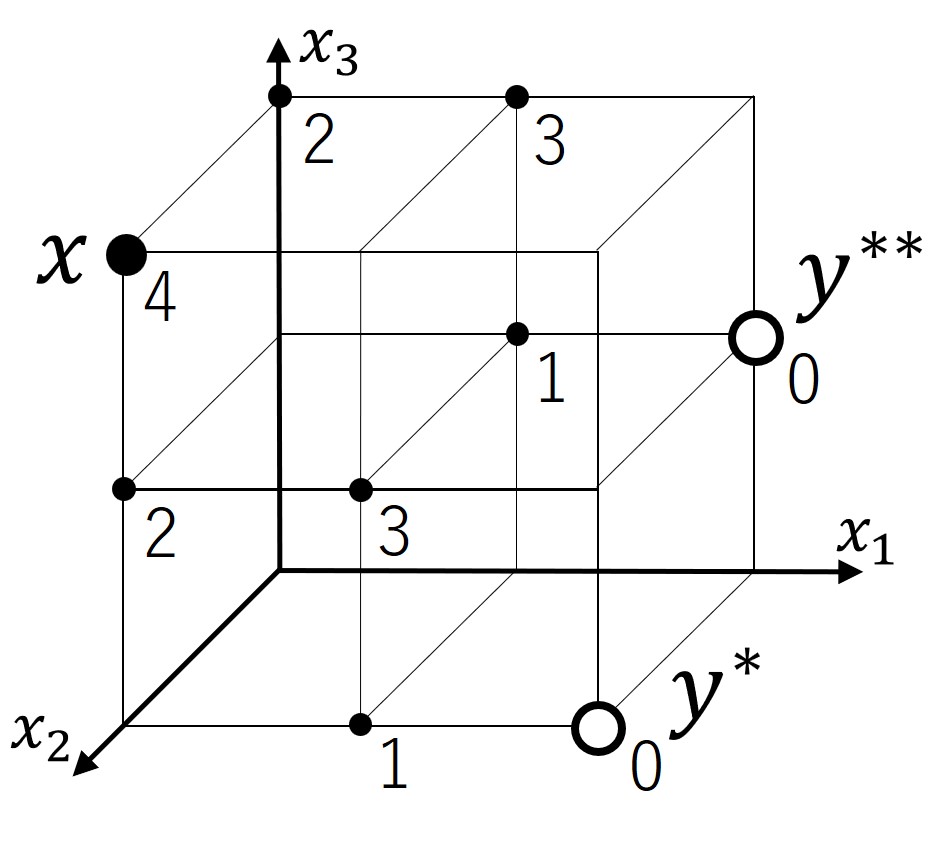}
 \caption{Values of function $f$
 in Example \ref{ex:min-cut-qMnat-1}.}
 \label{fig:geodesic-example2}
 \end{center}
 \end{figure}
 For $x =(0,1,2)$, the pair $(i,j) = (2, 0)$ is a valid choice
in Theorem \ref{thm:min-cut-qMnat}, since
we have
\begin{align*}
&  x - \chi_i + \chi_j = (0,0,2), 
\\
&  f(x - \chi_i + \chi_j) = 2 = \min_{i',j'\in \Nzero} f(x - \chi_{i'} + \chi_{j'}).
\end{align*} 
However,  neither of the two minimizers $y^* = (2,1,0)$  and $y^{**}=(2,0,1)$
satisfies the inequality  $x^*(N) \le x(N)-1$ in \eqref{eqn:mincut-cond-v-Mnat},
while $y^{**}=(2,0,1)$  satisfies
the inequality $x^{*}(i) \le x(i)-1 = 0$
in~\eqref{eqn:mincut-cond-v}.
\qed
\end{example}


\subsection{Geodesic Property} 
\label{sec:geodesic}

 Geodesic property of a function $f: \Z^n \to \Rinf$
means that
whenever a vector $x \in \dom f$ moves to 
a local minimizer $x' \in N(x)$ 
in an appropriately defined neighborhood $N(x)$ of $x$,
the distance $\|x^* - x\|$
to a nearest minimizer $x^*$ from the current solution $x$
decreases by $\|x' - x\|$, where $\|\cdot\|$ is an appropriately chosen norm.
 It is known that M-convex functions
have the geodesic property with respect to 
the L$_1$-norm
\cite[Corollary~4.2]{Shi22L1}, \cite[Theorem~2.4]{MS21steepM}. 
 We first point out that
\Mnat-convex functions do not enjoy
the geodesic property with respect to the L$_1$-norm.

  For a function $f: \Z^n \to \Rinf$ and a vector
$x \in \dom f$, we define
\begin{align}
\mu(x) & =  \min\{\|x^* - x \|_1  \mid x^* \in \arg\min f\},
\label{eqn:def-mu}
\\
M(x) & =  \{x^* \in \Z^n \mid x^* \in \arg\min f, \  \|x^* - x \|_1 = \mu(x)  \}.
\label{eqn:def-Mx}
\end{align}
 The following is an expected  plausible statement of
the geodesic property for \Mnat-convex functions 
with respect to the L$_1$-norm.

\medskip

\noindent
{\bf Statement A:} \quad
 Let $x \in \dom f$ be a vector that is not a minimizer of $f$.
 Also, let $(i,j)$ be a pair of distinct elements in $\Nzero$
 minimizing the value $f(x - \chi_i + \chi_j)$,
 and define 
 \begin{equation*}
 {M}'=
 \begin{cases}
 \{x^* \in M(x)
 \mid x^*(i) \le x(i)-1,\ x^*(j) \ge x(j)+1  \}
 & (\mbox{if }i, j \in N),\\
 \{x^* \in M(x) \mid x^*(i) \le x(i)-1 \}
 & (\mbox{if }i \in N,\ j=0),\\
 \{x^* \in M(x) \mid x^*(j) \ge x(j)+1  \}
 & (\mbox{if }i=0,\ j \in N).
 \end{cases}
 \end{equation*}
 {\rm (i)}
 There exists some $x^* \in M(x)$ that is contained in ${M}'$;
 we have ${M}' \ne \emptyset$, in particular.
 \\
 {\rm (ii)}
 It holds that
\begin{align*}
\mu(x - \chi_i + \chi_j) 
&= 
 \begin{cases}
 \mu(x)-2
 & (\mbox{if }i, j \in N),\\
 \mu(x)-1
 & (\mbox{if }i =0 \mbox{ or }j=0),
 \end{cases}
 \\
 M(x - \chi_i + \chi_j) &={M}'.
\end{align*} 
 
\medskip

However, neither (i) nor (ii) of 
Statement A is true for \Mnat-convex functions,
as shown by the following example.

\begin{example}\rm
\label{ex:geodesic-Mnat-L1}
 This  example  shows that Statement A is not true for \Mnat-convex functions.
 Consider a function $f: \Z^2 \to \Rinf$ given by
 \begin{align*}
 & \dom f = \{x \in \Z^2 \mid 0 \le x(1) \le 2,\ 0 \le x(2) \le 1\},\\
 & f(x) = 2 - x(1) \qquad (x \in \dom f)
 \end{align*}
 (see Figure \ref{fig:geo-example}),
for which  $\arg\min f = \{ (2,0),(2,1) \}$.
 Function $f$ satisfies the condition (\Mnat-EXC), and hence it is \Mnat-convex.
 \begin{figure}[t]
 \begin{center}
 \includegraphics[width=0.38\textwidth]{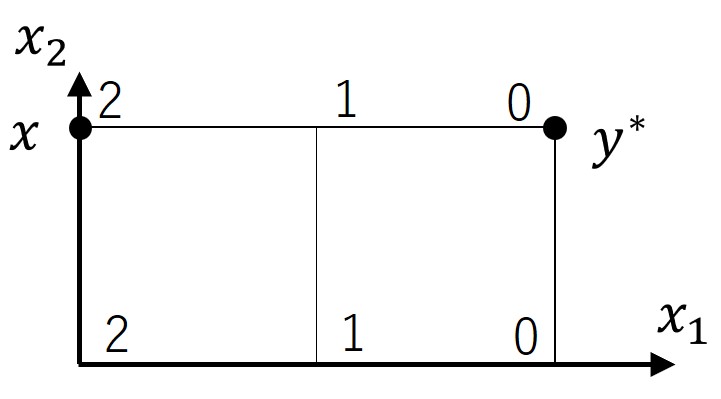}
 \caption{Values of function $f$
 in Example \ref{ex:geodesic-Mnat-L1}.}
 \label{fig:geo-example}
 \end{center}
 \end{figure}
 For $x =(0,1)$, we have
 \[
 \mu(x) = \|(2,1)- (0,1)\|_1 = 2, \qquad M(x) = \{(2,1)\}.
 \]
 We see that $(i,j)=(2,1)$ is a possible choice to minimize the value
 $f(x - \chi_i + \chi_j)$ among all $i,j \in \Nzero$.
 Then, we have
 \begin{align*}
 x' & = x - \chi_i + \chi_j = (1,0),
 \\
  {M}'
 & = M(x) \cap  \{x^*  \in \Z^2 \mid x^*(i) \le x(i)-1,\ x^*(j) \ge x(j)+1  \}\\
 & = \{(2,1)\}
 \cap \{x^* \in \Z^2 \mid x^*(2) \le 0,\ x^*(1) \ge 1  \} = \emptyset,
 \\
 \mu(x - \chi_i + \chi_j) 
 &= \|(2,0) -  (1,0)\|_1 = 1  \ne 0 =  \mu(x) - 2.
 \end{align*} 
 That is, Statement A does not hold for this \Mnat-convex function $f$.
\qed
\end{example}

 While Statement A fails for \Mnat-convex functions,
an alternative geodesic property 
holds for \Mnat-convex functions,
which can be obtained from
 the geodesic property for M-convex functions with respect to the L$_1$-norm
(see \cite[Corollary~4.2]{Shi22L1}, \cite[Theorem~2.4]{MS21steepM};
see also Theorem \ref{thm:geodesic-M} in Appendix).

 Let $f: \Z^n \to \Rinf$ be a function with $\arg\min f \ne \emptyset$.
 For $x \in \Z^n$, we define
\begin{align*}
\widetilde{\mu}(x) 
& =  \min\{\|x^* - x \|_1 + |x^*(N)-x(N)|  \mid x^* \in \arg\min f\},\\
\widetilde{M}(x) 
& =  \{x^* \in \Z^n \mid x^* \in \arg\min f, \  \|x^* - x \|_1  + |x^*(N)-x(N)|
= \widetilde{\mu}(x)  \}.
\end{align*}
That is, $\widetilde\mu(x)$ is a kind of distance from $x$ to
the nearest minimizer of $f$, and
$\widetilde{M}(x)$ is the set of the minimizers of $f$ nearest to $x$
with respect to this distance.

\begin{theorem}[{cf.~\cite[Corollary~4.2]{Shi22L1}, \cite[Theorem~2.4]{MS21steepM}}]
\label{thm:strong-min-cut-Mnat}
 Let $f: \Z^n \to \Rinf$ be an \Mnat-convex function with $\arg\min f \ne \emptyset$,
and $x \in \dom f$ be a vector that is not a minimizer of $f$,
i.e., $x \not\in \arg\min f$.
 Also, let $(i,j)$ be a pair of distinct elements in $\Nzero$
minimizing the value $f(x - \chi_i + \chi_j)$,
and define 
\begin{equation}
 \widetilde{M}'=
\begin{cases}
 \{x^* \in \widetilde{M}(x) 
\mid x^*(i) \le x(i)-1,\ x^*(j) \ge x(j)+1  \}
& (\mbox{if }i, j \in N),\\
 \{x^* \in \widetilde{M}(x)  \mid x^*(i) \le x(i)-1,\ 
x^*(N) \le x(N) -1 \}
& (\mbox{if }i \in N,\ j=0),\\
 \{x^* \in \widetilde{M}(x)  \mid x^*(j) \ge x(j)+1,\ 
x^*(N) \ge x(N) +1  \}
& (\mbox{if }i=0,\ j \in N).
 \end{cases}
\end{equation}
{\rm (i)}
 There exists some $x^* \in \widetilde{M}(x)$ that is contained 
in $\widetilde{M}'$;
we have $\widetilde{M}' \ne \emptyset$, in particular.
\\
{\rm (ii)}
It holds that
$\widetilde\mu(x - \chi_i + \chi_j) = 
\widetilde\mu(x)-2$
and
$\widetilde{M}(x - \chi_i + \chi_j) = \widetilde{M}'$.
\end{theorem}

 Note that the statement (i) immediately implies
the minimizer cut property (Theorem \ref{thm:min-cut-Mnat})
for \Mnat-convex functions.
 The statement (ii)
implies that in each iteration of Algorithm {\sc BasicSteepestDescent}
applied to an \Mnat-convex function,
the distance $\widetilde{\mu}(x)$ to the nearest minimizer reduces by two.
 This fact yields an exact number of iterations
required by {\sc BasicSteepestDescent}.

\begin{corollary}
 Let $f: \Z^n \to \Rinf$ be 
an \Mnat-convex function  with $\arg\min f \ne \emptyset$.
 Suppose that 
Algorithm {\sc BasicSteepestDescent} is applied to $f$
with the initial vector $x_0 \in \dom f$.
 Then, the number of iterations is equal to $\tilde{\mu}(x_0)/2$.
\end{corollary}

 In contrast, s.s.\;quasi \Mnat-convex functions do not enjoy
the  geodesic property in the form of Theorem~\ref{thm:strong-min-cut-Mnat},
as illustrated in the following example.

\begin{example}\rm 
\label{ex:geodesic-qMnat-1}
 This example shows that the
statement of Theorem \ref{thm:strong-min-cut-Mnat}
is not true for s.s.\;quasi \Mnat-convex functions
in the case of $i=0$ or $j=0$.
 Consider the s.s.\;quasi \Mnat-convex function $f: \Z^3 \to \Rinf$ 
in Example \ref{ex:min-cut-qMnat-1} (see  Figure \ref{fig:geodesic-example2}),
for which $\arg\min f = \{ (2,1,0),(2,0,1) \}$.
 For $x =(0,1,2)$,
we have
\[
 \widetilde{\mu}(x) = \|(2,1,0)- (0,1,2)\|_1
 = \|(2,0,1)- (0,1,2)\|_1 = 4, \quad
\widetilde{M}(x) = \{(2,1,0), (2,0,1)\}.
\]
 For $(i,j) = (2,0)$, we have $x - \chi_i + \chi_j=(0,0,2)$ 
and 
 \[
 f(x - \chi_i + \chi_j) = 2 = \min_{i',j'\in \Nzero} f(x - \chi_{i'} + \chi_{j'}).
 \]
 However, we have
\begin{align*}
 \widetilde{M}'
& = \widetilde{M}(x) \cap  
\{x^*  \in \Z^3 \mid x^*(i) \le x(i)-1,\ x^*(N) \le x(N) -1 \}\\
& = \{(2,1,0), (2,0,1)\} \cap 
 \{x^* \in \Z^3 \mid x^*(2) \le 0,\ x^*(N) \le 2 \}
 = \emptyset, 
\\
\widetilde{\mu}(x - \chi_i + \chi_j)
&= \|(2,0,1)- (0,0,2)\|_1 = 3 \ne 2 = \widetilde{\mu}(x) -2.
\end{align*}
Thus, the statements (i) and (ii)
in Theorem \ref{thm:strong-min-cut-Mnat} fail for $f$.
\qed
\end{example}


\subsection{Proximity Property} 
\label{sec:proximity}

 Given a function $f: \Z^n \to \Rinf$, $\hat{x} \in \dom f$,
and an integer $\alpha \ge 2$, we consider the following scaled
minimization problem:
\[
 \mbox{(SP) \quad Minimize }\quad  f(x) \qquad \mbox{ subject to } \quad
x= \hat{x} + \alpha y,\ y \in \Z^n.
\]
 It is expected that
an appropriately chosen neighborhood of a global (or local) optimal solution
of the scaled minimization problem (SP) contains
some minimizer of $f$; such a property is referred to as 
a proximity property in this paper.
 It is known that M-convex functions enjoy
a proximity property \cite[Theorem 3.4]{MMS02}
(see also Theorem \ref{thm:proxmity-M} in Appendix),
which can be rewritten in terms of \Mnat-convex functions 
as follows:

\begin{theorem}[{cf.~\cite[Theorem 3.4]{MMS02}}]
\label{thm:proxmity-Mnat}
 Let $f: \Z^n \to \Rinf$ be an \Mnat-convex function 
and $\alpha\ge 2$ be an integer.
 For every vector ${x} \in \dom f$  satisfying
\begin{equation*}
  f({x})  \le 
 \min\Big[
 \min_{i\in N}f({x}\pm \alpha \chi_i),\ 
 \min_{i,j \in N}f({x} - \alpha (\chi_i -  \chi_j))
 \Big],
\end{equation*} 
 there exists some minimizer $x^*$ of $f$ satisfying
\[
 \|x^* - {x} \|_\infty \le n (\alpha -1), \qquad
|x^*(N) - {x}(N)| \le n (\alpha -1).
\]
\end{theorem}
\noindent
 This theorem, in particular, implies that
for every optimal solution ${x}$ of {\rm (SP)},
 there exists some minimizer $x^*$ of $f$ such that
$\|x^* - {x} \|_\infty \le n (\alpha -1)$.

 In contrast, a proximity property of this form
does not hold for s.s.\;quasi \Mnat-convex functions.
 Indeed, the following example shows that there exists a family of
s.s.\;quasi \Mnat-convex functions such that
the distance between an approximate global minimizer $\hat{x}$
and a unique exact global minimizer $x^*$ can be arbitrarily large.

\begin{example}\rm
\label{ex:proximity}
 This example shows a function  satisfying
{\SSQMn} for which
the statement of Theorem \ref{thm:proxmity-Mnat} does not hold.
 With an integer $k \ge 2$,
define a function $f: \Z^3 \to \Rinf$ as follows
(see Figure \ref{fig:func} for the case of $k=3$):
 \begin{align*}
 &  \dom f = \{x \in \Z^3 \mid
 0 \le x_1 \le k,\ 0 \le x_2 \le 1,\ 0 \le x_3 \le 1  \},\\
 & f(\lambda, 0,0) = 0 \qquad (0 \le \lambda \le k),\\
 & f(\lambda, 1,0) = f(\lambda, 0,1) = \lambda-k-1 \qquad (0 \le \lambda \le k),\\
 & f(\lambda, 1,1) = 2(\lambda-k-1) \qquad (0 \le \lambda \le k).
 \end{align*}
 It can be verified that $f$ satisfies the condition {\SSQMn}.
 Note that the value $f(\lambda,0,0)$ is constant
for every $\lambda$ with $0 \le \lambda \le k$,
while $f(\lambda,1,0), f(\lambda,0,1)$, and $f(\lambda,1,1)$
are strictly increasing with respect to $\lambda$
in the interval $0 \le \lambda \le k$.
 We also see that $f$ has a unique minimizer $y^*=(0,1,1)$. 
 \begin{figure}[t]
 \begin{center}
 \includegraphics[width=0.6\textwidth]{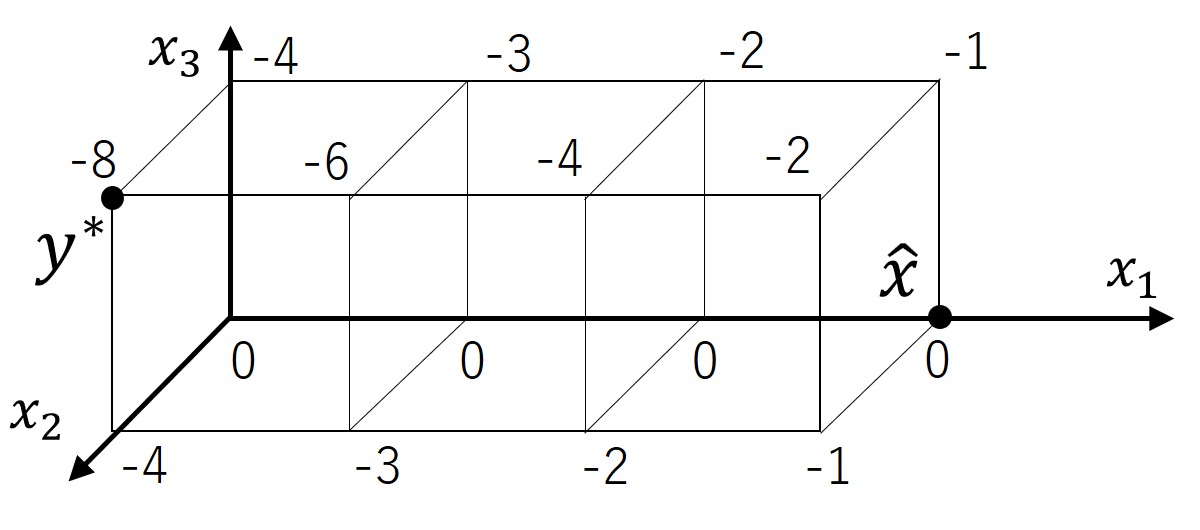}
 \caption{Values of function $f$
 in Example \ref{ex:proximity} with $k=3$.}
 \label{fig:func}
 \end{center}
 \end{figure}
For the problem (SP) with $\hat{x} = (k,0,0)$ and $\alpha = 2$,
$\hat{x}$ itself is an optimal solution.
 The distance between the unique minimizer $y^*$ of $f$
and $\hat{x}$ is $\|\hat{x} - y^* \|_\infty = k$,
which can be arbitrarily large by taking a sufficiently large $k$.
\qed
\end{example}


\section{Proofs}
\label{sec:proof}

\subsection{Minimizer Cut Property}
\label{sec:proof-min-cut}

 We derive Theorem \ref{thm:min-cut-qMnat} from the following
variants of the minimizer cut property for
s.s.\;quasi \Mnat-convex functions.

\begin{theorem}
\label{thm:min-cut-qMnat-2}
 Let $f: \Z^n \to \Rinf$ be a function 
with {\SSQMn} satisfying $\arg\min f \ne \emptyset$,
and $x \in \dom f$.
\\
{\rm (i)}
Let $i \in N$ and suppose that the minimum of
$f(x - \chi_{i} + \chi_{j'})$ over $j' \in \Nzero$
is attained by some $j \in N$ $(j \ne 0)$.
 Then,
there exists some minimizer $x^*$ of $f$ satisfying
\begin{equation*}
\begin{cases}
x^*(j) \ge x(j)+1 
& (\mbox{if }j \in N \setminus \{i\}),\\
x^*(i) \ge x(i) 
& (\mbox{if }j =i).
 \end{cases}
\end{equation*}
{\rm (ii)}
Symmetrically to {\rm (i)},
let $j \in N$ and suppose that the minimum of
$f(x - \chi_{i'} + \chi_{j})$ over $i' \in \Nzero$
is attained by some $i \in N$ $(i \ne 0)$.
Then, there exists some minimizer $x^*$ of $f$ satisfying
\begin{equation*}
\begin{cases}
x^*(i) \le x(i)-1 
& (\mbox{if }i \in N \setminus \{j\}),\\
x^*(j) \le x(j) 
& (\mbox{if }i =j).
 \end{cases}
\end{equation*}
{\rm (iii)}
 Suppose that the minimum of
$f(x  + \chi_{j'})$ over $j' \in \Nzero$
is attained by some $j \in N$ $(j \ne 0)$.
 Then,
there exists some minimizer $x^*$ of $f$ satisfying
$x^*(j) \ge x(j)+1$.
\\
{\rm (iv)}
Symmetrically to {\rm (iii)},
suppose that the minimum of
$f(x  - \chi_{i'})$ over $i' \in \Nzero$
is attained by some $i \in N$ $(i \ne 0)$.
Then, 
there exists some minimizer $x^*$ of $f$ satisfying
$x^*(i) \le x(i)-1$.
\end{theorem}
\noindent
 While postponing the proof of Theorem \ref{thm:min-cut-qMnat-2},
we first give a proof of Theorem \ref{thm:min-cut-qMnat}.

\begin{proof}[Proof of Theorem \ref{thm:min-cut-qMnat}]
 We first consider the case of $j \in N$ (i.e., $j \ne 0$ and $i \in \Nzero$).
 By the choice of $j$, we have 
$f(x - \chi_i + \chi_j) = \min_{j' \in \Nzero}
f(x - \chi_i + \chi_{j'})$. 
 Hence, we can apply Theorem \ref{thm:min-cut-qMnat-2} (i) 
(if $i \ne 0$)
or Theorem \ref{thm:min-cut-qMnat-2} (iii) (if $i =0$) to obtain
some $x^* \in \arg\min f$ such that $x^*(j) \ge x(j)+1$.
 If $i = 0$ then we are done since $x^*$ satisfies 
the desired condition~\eqref{eqn:mincut-cond-v}.

  We consider the case $i\ne 0$.
  Let $\tilde{f}: \Z^n \to \Rinf$ be
the restriction of $f$ to
  $D = \{y \in \Z^n \mid y(j) \ge x(j) + 1\}$.
Since
$x^* \in D \cap \arg\min f$,
it holds that
\[
\min \tilde{f} = \min f, \qquad
\arg\min \tilde{f} \subseteq \arg\min f.
\]
 We can check easily that $\tilde{f}$ satisfies {\SSQMn}
and $i \in N$ satisfies
\[
\tilde{f}(x - \chi_i + \chi_j)
 = \min_{i' \in \Nzero} \tilde{f}(x - \chi_{i'} + \chi_{j}).
\]
 Hence, we can apply Theorem \ref{thm:min-cut-qMnat-2} (ii) to obtain 
some minimizer $x^{**}$ of $\tilde{f}$ satisfying
$x^{**}(i) \le x(i)-1$.
 This vector $x^{**}$ satisfies 
$x^{**}(i) \le x(i)-1$ and $x^{**}(j) \ge x(j)+1$ as
desired in \eqref{eqn:mincut-cond-v}.
 Note that $x^{**}$ is a minimizer of $f$ because 
$x^{**} \in \arg\min \tilde{f} \subseteq \arg\min f$.

 The remaining case (i.e., $i \in N$ and $j = 0$)
can be treated similarly
by using Theorem~\ref{thm:min-cut-qMnat-2}~(iv).
\end{proof}

 We now give a proof of Theorem \ref{thm:min-cut-qMnat-2}.

\begin{proof}[Proof of Theorem \ref{thm:min-cut-qMnat-2}]
 In the following, we give proofs of (i) and (iii);
proofs of 
(ii) and (iv) can be obtained by applying (i)
and (iii) to $g(x) = f(-x)$, 
respectively.

[Proof of (i)] \quad
 Put $x' =x - \chi_i + \chi_j$.
 It suffices to show that $x^*(j) \ge x'(j)$ holds
for some $x^* \in \arg\min f$.
 Let $x^*$ be a vector in $\arg\min f$
that maximizes the value $x^*(j)$.
 If $x^*$ satisfies $x^*(j) \ge x'(j)$, then we are done.
 Hence, we assume, to the contrary, that
$x^*(j) < x'(j)$, and derive a contradiction.

 The condition {\SSQMn} applied to
$x'$, $x^*$, and $j \in \suppp(x' - x^*)$ implies that
there exists some  $r \in \suppm(x' - x^*)\cup\{0\}$ such that
\begin{align}
&  f(x^*) >  f(x^* + \chi_j - \chi_r),
\quad \mbox{ or } 
\label{eqn:SSQMcond2-1}\\
& f(x') >  f(x' - \chi_j + \chi_r),
\quad \mbox{ or } 
\label{eqn:SSQMcond2-2}\\
 & f(x^*) =  f(x^* + \chi_j - \chi_r)
 \mbox{ and } 
f(x') =  f(x' - \chi_j + \chi_r).
\label{eqn:SSQMcond2-3}
\end{align}
 Note that $r \ne j$ holds.
 Since $x^* \in \arg\min f$, we have
\begin{equation}
\label{eqn:mincut-9} 
  f(x^*) \le  f(x^* + \chi_j - \chi_r).
\end{equation}
 By the choice of $j$ and $x' =x - \chi_i + \chi_j$, we have
\begin{equation}
\label{eqn:mincut-10} 
 f(x') \le  f(x - \chi_i + \chi_r) = f(x' - \chi_j + \chi_r),
\end{equation}
where $x - \chi_i + \chi_r = x' - \chi_j + \chi_r$.
 The inequalities \eqref{eqn:mincut-9} and \eqref{eqn:mincut-10} 
exclude the possibilities of \eqref{eqn:SSQMcond2-1} and \eqref{eqn:SSQMcond2-2},
respectively.
 Therefore, we have \eqref{eqn:SSQMcond2-3}.
 The former equation in \eqref{eqn:SSQMcond2-3}
implies that $x^* + \chi_j - \chi_r$ is also a minimizer of $f$,
a contradiction to the choice of $x^*$
since $(x^* + \chi_j - \chi_r)(j)  = x^*(j)+1 > x^*(j)$.

[Proof of (iii)] \quad
 The proof given below is similar to that for (i).
 Put $x' =x  + \chi_j$.
 It suffices to show that $x^*(j) \ge x'(j)$ holds
for some $x^* \in \arg\min f$.
 Let $x^*$ be a vector in $\arg\min f$
that maximizes the value $x^*(j)$.
 If $x^*$ satisfies $x^*(j) \ge x'(j)$, then we are done.
 Hence, we assume, to the contrary, that
$x^*(j) < x'(j)$, and derive a contradiction.

 The condition {\SSQMn} applied to
$x'$, $x^*$, and $j \in \suppp(x' - x^*)$ implies that
there exists some  $r \in \suppm(x' - x^*)\cup\{0\}$ such that
\begin{align}
&  f(x^*) >  f(x^* + \chi_j - \chi_r),
\quad \mbox{ or } 
\label{eqn:SSQMcond2-1-ii}\\
& f(x') >  f(x' - \chi_j + \chi_r),
\quad \mbox{ or } 
\label{eqn:SSQMcond2-2-ii}\\
 & f(x^*) =  f(x^* + \chi_j - \chi_r)
 \mbox{ and } 
f(x') =  f(x' - \chi_j + \chi_r).
\label{eqn:SSQMcond2-3-ii}
\end{align}
 Note that $r \ne j$ holds.
 Since $x^* \in \arg\min f$, we have
\begin{equation}
\label{eqn:mincut-9-ii} 
  f(x^*) \le  f(x^* + \chi_j - \chi_r).
\end{equation}
 By the choice of $j$ and $x' =x + \chi_j$, we have
\begin{equation}
\label{eqn:mincut-10-ii} 
 f(x') \le  f(x + \chi_r) = f(x' - \chi_j + \chi_r),
\end{equation}
where $x + \chi_r = x' - \chi_j + \chi_r$.
 The inequalities \eqref{eqn:mincut-9-ii} and \eqref{eqn:mincut-10-ii} 
exclude the possibilities of \eqref{eqn:SSQMcond2-1-ii} and \eqref{eqn:SSQMcond2-2-ii},
respectively.
 Therefore, we have \eqref{eqn:SSQMcond2-3-ii}.
 The former equation in \eqref{eqn:SSQMcond2-3-ii}
implies that $x^* + \chi_j - \chi_r$ is also a minimizer of $f$,
a contradiction to the choice of $x^*$
since $(x^* + \chi_j - \chi_r)(j)  = x^*(j)+1 > x^*(j)$.
\end{proof}


\subsection{Domain Reduction Algorithm}
\label{sec:domain_reduction}

 To prove Corollary \ref{cor:domain-reduction-quasiMnat},
we need to explain the general framework of the domain reduction 
algorithm for minimization
of a function $f: \Z^n \to \Rinf$ with bounded $\dom f$.
  For a nonempty bounded set $S \subseteq \Z^n$,
we denote
\begin{align*}
& \ell(S; i) = \min\{x(i) \mid x \in S \},
\qquad u(S; i) = \max\{x(i) \mid x \in S \} \qquad (i \in N).
\end{align*}
 We define the \textit{peeled set} $\widehat{S}$ of $S$
by
\begin{align*}
 \widehat{S} 
&= \{x \in S \mid \ell'(i) \le x(i) \le u'(i) \ (i \in N)\},\\
 \ell'(i) &= (1-1/n)\ell(S;i) + (1/n)u(S;i) \qquad (i \in N), 
\\
u'(i) & = (1/n)\ell(S;i) + (1-1/n)u(S;i) \qquad (i \in N).
\end{align*}

 The outline of the algorithm is described as follows.

\begin{flushleft}
 \textbf{Algorithm} {\sc DomainReduction}
\\
 \textbf{Step 0:} 
 Let $B:= \dom f$.
\\
 \textbf{Step 1:} 
 Find a vector $x$ in the peeled set $\widehat{B}$.
\\
 \textbf{Step 2:} 
 If $x$ is a minimizer of $f$, then output $x$ and stop.
\\
 \textbf{Step 3:} 
 Find an axis-orthogonal hyperplane $y(i) = \alpha$ 
with some $i \in N$ and $\alpha \in \Z$
	  such that 
\begin{quote}
\qquad   (Case 1) 
 $\arg\min f \cap \{y \mid y(i) \ge \alpha \} \ne \emptyset$
 and $x \notin \{y \mid y(i) \ge \alpha \}$
 \quad or \\
\qquad (Case 2) 
 $\arg\min f \cap \{y \mid y(i) \le \alpha \} \ne \emptyset$
 and $x \notin \{y \mid y(i) \le \alpha \}$.
\end{quote}
 \textbf{Step 4:} 
 Set 
\[
 B := 
\begin{cases}
B \cap \{y \mid y(i) \ge \alpha\} 
& (\mbox{Case 1}),
\\
B \cap \{y \mid y(i) \le \alpha\} 
& (\mbox{Case 2})
\end{cases}
\]
\phantom{\textbf{Step 4:}\ }
and go to Step 1.
\end{flushleft}

The algorithm successfully finds a minimizer of $f$ if the following are true 
in each iteration:
\begin{quote}
{(DR1)}
the peeled set $\widehat{B}$ in Step 1 is nonempty, and
\\
{(DR2)}
there exists a hyperplane satisfying the desired condition 
in Step 3.
\end{quote}
These conditions hold indeed 
if $f$ is an s.s.\;quasi \Mnat-convex function,
as we show later (after Lemma \ref{lem:domain_reduction}).

The number of iterations of the algorithm {\sc DomainReduction}
can be analyzed as follows.

\begin{lemma}[{cf.~\cite[Section 3]{Shi98min}}]
\label{lem:domain_reduction}
 If the conditions {\rm (DR1)} and {\rm (DR2)}
are satisfied in each iteration of {\sc DomainReduction},
then the algorithm terminates in $O(n^2 \log L)$ iterations
with $L = L_\infty(\dom f)$.
\end{lemma}

\begin{proof}
 We provide a proof for completeness.
 We first note that the set $B$ always contains a minimizer of $f$.
 For an iteration of the algorithm,
we say that it is \textit{of type $i$} if the hyperplane found in
Step 3 is of the form $y(i)=\alpha$.
 In an iteration of type $i$,
the value $u(B;i)- \ell(B;i)$ decreases by 
$(1/n)(u(B;i)- \ell(B;i))$
by the choices of the vector $x$ in Step 1 and the hyperplane in Step 3.
 This implies that after $O(n \log L)$ iterations
of type $i$, we have
$u(B;i)- \ell(B;i) < 1$, i.e., $u(B;i)= \ell(B;i)$,
and therefore an iteration of type $i$ never occurs afterwards.
 Hence, after $O(n^2 \log L)$ iterations we have
$u(B;i)= \ell(B;i)$ for all $i \in N$,
i.e., $B$ consists of a single vector, which must be a minimizer of $f$.
\end{proof}

We now assume that
 $f$ is an s.s.\;quasi \Mnat-convex function (i.e., satisfies {\SSQMn})
such that the effective domain of $f$ is a bounded \Mnat-convex set,
and show that the conditions (DR1) and (DR2)
are satisfied in each iteration.
 The minimizer cut property
(Theorem \ref{thm:min-cut-qMnat})
guarantees the condition (DR2) for
an s.s.\;quasi \Mnat-convex function.
An axis-orthogonal hyperplane satisfying the condition in Step 3
can be found by evaluating function values $O(n^2)$ times.

The condition (DR1), i.e.,
the nonemptiness of the peeled set $\widehat{B}$, can be shown as follows.
 We say that a set $S \subseteq \Z^n$ is an \Mnat-convex set
(see, e.g., \cite[Section 4.7]{Mdcasiam})
if it satisfies the following exchange axiom:
\begin{quote}
 $\forall x, y \in S$, $\forall i \in \suppp (x - y)$,
$\exists j \in \suppm(x - y)\cup\{0\}$ such that
\begin{equation}
\label{eqn:def-Mnat-conv-set}
 x - \chi_{i} + \chi_{j} \in S, \qquad
y + \chi_{i} - \chi_{j} \in S.
\end{equation}
\end{quote}
 It is known that for an \Mnat-convex set $S \subseteq \Z^n$
and 
an integer interval $[a, b]\ (\subseteq \Z^n)$,
their intersection $S \cap [a,b]$
is again an \Mnat-convex set if it is nonempty.
 Hence, the set $B$ in each iteration of the algorithm
is always an \Mnat-convex set as far as it is nonempty.
 The following lemma shows that the set $B$ is always nonempty.

\begin{lemma}
 For a bounded \Mnat-convex set $B \subseteq \Z^n$,
the peeled set $\widehat{B}$ is nonempty.
\end{lemma}

\begin{proof}
The proof can be reduced to a similar statement 
known for an M-convex set
(see, e.g., \cite{Mdcasiam} for the definition of M-convex set).
First note that a set $B \subseteq \Z^n$ is an \Mnat-convex set
if and only if the set 
$S = \{(y, -y(N)) \mid y \in B\}$
$(\subseteq \Z^n \times \Z)$ 
is an M-convex set.
 For an \Mnat-convex set $B \subseteq \Z^n$,
the peeled set $\widehat{S}$  of the associated M-convex set
$S$ is nonempty by \cite[Theorem 2.4]{Shi98min},
whereas
$\{y \mid (y, -y(N)) \in \widehat{S}\}  \subseteq \widehat{B}$.
Therefore, we have $\widehat{B} \ne \emptyset$.
\end{proof}
\noindent
 We note that for a bounded \Mnat-convex set $B \subseteq \Z^n$,
a vector in the peeled set $\widehat{B}$ can be found
in $O(n^2 \log L)$ time with $L = L_\infty(\dom f)$
as in \cite{Shi98min}.
 Hence,  the condition (DR1) is also satisfied for
s.s.\;quasi \Mnat-convex functions.
 By Theorem \ref{thm:Mminimizer-SSQMn}, Step 2 can be done in
$O(n^2)$ time.
 This concludes the proof of
Corollary \ref{cor:domain-reduction-quasiMnat}.


\section{Concluding Remarks}

\subsection{Remarks on Minimizer Cut Property}
\label{sec:min-cut-remark}

 We note that, 
in addition to Theorem \ref{thm:min-cut-Mnat}, 
\Mnat-convex functions enjoy
the following variants of the minimizer cut property,
which are similar to, but 
stronger than, the statements of Theorem \ref{thm:min-cut-qMnat-2}
for s.s.\;quasi \Mnat-convex functions.
 The statements in the theorem below can be obtained from
the corresponding statements
for M-convex functions \cite[Theorem 2.2]{Shi98min}
(see Theorem \ref{thm:min-cut-M} (i), (ii) in Appendix).

\begin{theorem}[{cf.~{\cite[Theorem 2.2]{Shi98min}}}]
\label{thm:min-cut-Mnat-2}
 Let $f: \Z^n \to \Rinf$ be an \Mnat-convex function with $\arg\min f \ne \emptyset$,
and $x \in \dom f$.
\\
{\rm (i)}
 Let $i \in N$ and let $j$ be an element in $\Nzero$
minimizing the value $f(x - \chi_i + \chi_j)$.
 Then, there exists some minimizer $x^*$ of $f$ satisfying
\begin{equation*}
\begin{cases}
x^*(j) \ge x(j)+1 
& (\mbox{if }j \in N \setminus \{i\}),\\
x^*(i) \ge x(i) 
& (\mbox{if }j =i),\\
x^*(N) \le x(N) -1
& (\mbox{if }j =0).
 \end{cases}
\end{equation*}
{\rm (ii)}
Symmetrically to {\rm (i)},
 let $j \in N$ and let $i$ be an element in $\Nzero$
minimizing the value $f(x - \chi_i + \chi_j)$.
 Then, there exists some minimizer $x^*$ of $f$ satisfying
\begin{equation*}
\begin{cases}
x^*(i) \le x(i)-1 
& (\mbox{if }i \in N \setminus \{j\}),\\
x^*(j) \le x(j) 
& (\mbox{if }i =j),\\
x^*(N) \ge x(N) +1
& (\mbox{if }i =0).
 \end{cases}
\end{equation*}
{\rm (iii)}
Let $j$ be an element in $\Nzero$
minimizing the value $f(x  + \chi_j)$.
 Then, there exists some minimizer $x^*$ of $f$ satisfying
\begin{equation*}
\begin{cases}
x^*(j) \ge x(j)+1 
& (\mbox{if }j \in N),\\
x^*(N) \le x(N)
& (\mbox{if }j =0).
 \end{cases}
\end{equation*}
{\rm (iv)}
Symmetrically to {\rm (iii)},
let $i$ be an element in $\Nzero$
minimizing the value $f(x - \chi_i)$.
 Then, there exists some minimizer $x^*$ of $f$ satisfying
\begin{equation*}
\begin{cases}
x^*(i) \le x(i)-1 
& (\mbox{if }i \in N),\\
x^*(N) \ge x(N)
& (\mbox{if }i =0).
 \end{cases}
\end{equation*}
\end{theorem}

It is in order here to dwell on the difference between
Theorem \ref{thm:min-cut-Mnat-2} 
and Theorem~\ref{thm:min-cut-qMnat-2}
by focusing on the statement (i).
The statement (i) of
Theorem \ref{thm:min-cut-Mnat-2} 
for \Mnat-convex functions
covers all possible cases of $j \in \Nzero$.
In contrast,
the statement (i) of
Theorem~\ref{thm:min-cut-qMnat-2} 
for s.s.~quasi \Mnat-convex functions
puts an assumption that $j \ne 0$ attains the minimum,
which means that we can obtain no conclusion 
if the minimum of
$f(x - \chi_{i} + \chi_{j})$ over $j \in \Nzero$
is attained uniquely by $j = 0$.
Thus, the statement (i) of
Theorem~\ref{thm:min-cut-qMnat-2} 
is strictly weaker than
the statement (i) of
Theorem \ref{thm:min-cut-Mnat-2}.

 We present an example
 to show that  the statements (i) and  (iii) of
Theorem \ref{thm:min-cut-Mnat-2} in the case of $j=0$
do not hold for s.s.\;quasi \Mnat-convex functions.

 \begin{figure}[t]
 \begin{center}
 \includegraphics[width=0.4\textwidth]{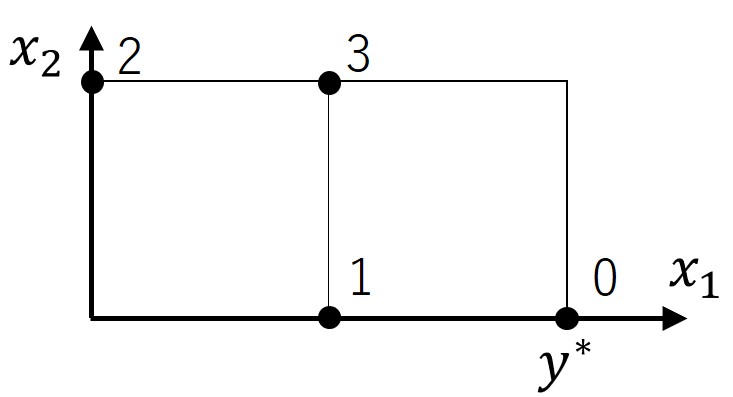}
 \caption{Values of function $f$
 in Example \ref{ex:mincut-counter-ex}.} 
 \label{fig:mincut-example}
 \end{center}
 \end{figure}

\begin{example}\rm
\label{ex:mincut-counter-ex}
 Here is an example to show that 
the statements (i) and  (iii) of
Theorem \ref{thm:min-cut-Mnat-2} in the case of $j=0$
are not true for s.s.\;quasi \Mnat-convex functions.
 Consider the function $f: \Z^2 \to \Rinf$ given by%
\footnote{
 This function $f$ is an adaptation of the one in \cite[Example 5.1]{MY15mor}
to a discrete quasi convex function.
 The function in \cite[Example 5.1]{MY15mor}
is used in \cite{MY15mor}
to point out  the difference between the weaker versions of 
{\SSQMn} and (SSQ\Mnat-EXC-PRJ);
see also Section \ref{sec:quasiM}.
}
\begin{align*}
&  \dom f= \{(1,0), (2,0), (0,1), (1,1)\},\\
& f(1,0) = 1,\ f(2,0)=0,\ f(0,1)=2,\ f(1,1)=3
\end{align*}
(see Figure \ref{fig:mincut-example}).
 Function $f$ satisfies the condition {\SSQMn}.

 For $x =(1,1)$ and $i = 1$, 
$j = 0$ minimizes the value $f(x - \chi_1 + \chi_j)$
among all $j \in \Nzero$.
 However, the unique minimizer $y^* = (2,0)$ 
does not satisfy $y^*(N) \le x(N)-1$,
i.e.,
the statement (i) of
Theorem \ref{thm:min-cut-Mnat-2} does not hold
in the case of $j=0$.

 For $y = (0,1)$,
$j=0$ minimizes  the value $f(y + \chi_j)$
among all $j \in \Nzero$.
 However, the unique minimizer $y^* = (2,0)$ 
does not satisfy $y^*(N) \le y(N)$, i.e.,
the statement (iii) of
Theorem \ref{thm:min-cut-Mnat-2} does not hold
in the case of $j=0$.
 \qed
\end{example}

 We can also show that the statements (ii) and (iv) of 
Theorem~\ref{thm:min-cut-qMnat-2} in the case of $i=0$ 
do not hold for s.s.\;quasi \Mnat-convex functions;
a counterexample to the statements is given by
the function $g(x) = f(-x)$ with
the function $f$ in Example \ref{ex:mincut-counter-ex}.


\subsection{Connection with Quasi M-convex Functions}
\label{sec:quasiM}

As mentioned in Introduction, the concept of
s.s.\;quasi M-convex function is proposed
in \cite{Mdcasiam,MS03quasi} as a quasi-convex version of 
M-convex function.
 We explain the subtle difference between
s.s.\;quasi M-convexity and s.s.\;quasi \Mnat-convexity in this section%
\footnote{
The difference between the weaker versions of 
{\SSQMn} and (SSQ\Mnat-EXC-PRJ) is already pointed out by
Murota and Yokoi \cite{MY15mor}.
}.

 Recall that the condition {\SSQMn}
defining s.s.\;quasi \Mnat-convexity is obtained 
by relaxing the condition (\Mnat-EXC) for \Mnat-convex functions.
 Similarly, the concept of s.s.\;quasi M-convex function
is defined by using the relaxed version of the exchange axiom
for M-convex functions as follows.

A function $f: \Z^n \to \Rinf$ is said to be M-convex
if it satisfies the following exchange axiom:
\begin{quote}
{\bf (M-EXC)}
 $\forall x, y \in \dom f$, $\forall i \in \suppp (x - y)$,
$\exists j \in \suppm(x - y)$ such that
\begin{equation}
\label{eqn:def-Mnat-ineq-2}
  f(x) + f(y) \geq f(x - \chi_{i} + \chi_{j}) 
 + f(y + \chi_{i} - \chi_{j}).
\end{equation}
\end{quote}
 A semi-strictly quasi M-convex function 
is defined as a function satisfying 
the following relaxed version of (M-EXC):
\begin{quote}
{\SSQMb}
 $\forall x, y \in \dom f$, $\forall i \in \suppp (x - y)$,
$\exists j \in \suppm(x - y)$ satisfying
at least one of the three conditions:
\begin{align}
& f(x - \chi_i + \chi_j)< f(x),
\label{eqn:SSQM-1}\\
& f(y + \chi_i - \chi_j) < f(y),
\label{eqn:SSQM-2}\\
& f(x - \chi_i + \chi_j) = f(x) \mbox{ and } f(y + \chi_i - \chi_j) = f(y).
\label{eqn:SSQM-3}
\end{align}
\end{quote}

 Recall also that an \Mnat-convex function is originally
defined as the projection of an M-convex function:
an \Mnat-convex function is defined as a function
$f: \Z^n \to \Rinf$ 
such that the function $\tilde{f}: \Z^n \times \Z \to \Rinf$ given by
\begin{equation}
\label{eqn:def-tildef}
  \tilde{f}(x, x_0) =
\begin{cases}
f(x) & (\mbox{if }x_0= - x(N)),\\
+ \infty & (\mbox{if }x_0 \ne - x(N))
\end{cases}
\end{equation}
is M-convex (i.e., satisfies (M-EXC)).
 Hence, a function $f$ is \Mnat-convex if and only if
it satisfies the following exchange axiom obtained by the projection of (M-EXC):
\begin{quote}
{\bf (\Mnat-EXC-PRJ)}
 $\forall x, y \in \dom f$, \\
(i) if $x(N) > y(N)$, then 
$\forall i \in \suppp (x - y)$,
$\exists j \in \suppm(x - y)\cup\{0\}$ satisfying
\eqref{eqn:def-Mnat-ineq-2},\\
(ii) if $x(N) \le y(N)$, then 
$\forall i \in \suppp (x - y)$,
$\exists j \in \suppm(x - y)$ satisfying
\eqref{eqn:def-Mnat-ineq-2},\\
(iii) if $x(N) < y(N)$, then 
$\exists j \in \suppm(x - y)$ satisfying
\eqref{eqn:def-Mnat-ineq-2} with $i=0$.
\end{quote}
 That is, the following equivalence holds for \Mnat-convex functions.
\begin{theorem}[{\cite[Theorem 4.2]{MS99gp}}]
\label{thm:equiv-Mnat}
 For $f: \Z^n \to \Rinf$, 
\begin{equation}
\label{eqn:Mnat-thm-equiv}
\tilde{f} \mbox{ in \eqref{eqn:def-tildef} satisfies {\rm (M-EXC)}}
\iff
f \mbox{ satisfies {\rm (\Mnat-EXC-PRJ)}}
\iff
f \mbox{ satisfies {\rm (\Mnat-EXC)}.}
\end{equation}
\end{theorem}

For s.s.\;quasi \Mnat-convex functions,
we may similarly consider the function $\tilde{f}$ in \eqref{eqn:def-tildef}  
and the projected version of {\SSQM}:
\begin{quote}
{\bf (SSQ\Mnat-EXC-PRJ)}
 $\forall x, y \in \dom f$, \\
(i) if $x(N) > y(N)$, then 
$\forall i \in \suppp (x - y)$,
$\exists j \in \suppm(x - y)\cup\{0\}$ satisfying
at least one of \eqref{eqn:SSQM-1}, \eqref{eqn:SSQM-2}, and \eqref{eqn:SSQM-3},
\\
(ii) if $x(N) \le y(N)$, then 
$\forall i \in \suppp (x - y)$,
$\exists j \in \suppm(x - y)$ satisfying
at least one of \eqref{eqn:SSQM-1}, \eqref{eqn:SSQM-2}, and \eqref{eqn:SSQM-3},
\\
(iii) if $x(N) < y(N)$, then 
$\exists j \in \suppm(x - y)$ satisfying
at least one of \eqref{eqn:SSQM-1}, \eqref{eqn:SSQM-2}, and \eqref{eqn:SSQM-3}
 with $i=0$.
\end{quote}

\noindent
 While the conditions (\Mnat-EXC-PRJ) and (\Mnat-EXC)
are equivalent to each other, 
as mentioned in Theorem \ref{thm:equiv-Mnat},
the condition (SSQ\Mnat-EXC-PRJ) is 
not equivalent to, but strictly stronger than, {\SSQMn}
for s.s.\;quasi \Mnat-convex functions.
That is, we have
\begin{equation}
\label{eqn:SSQMn-equiv}
\mbox{$\tilde{f}$ in \eqref{eqn:def-tildef} satisfies {\SSQM}
$\iff$
$f$ satisfies {\rm (SSQ\Mnat-EXC-PRJ)}
\renewcommand{\arraystretch}{.4}
$\begin{array}{l}
\Longrightarrow \\
\Longleftarrow\!\!\!\!\!\!\not
\end{array}$
\renewcommand{\arraystretch}{1}
$f$ satisfies {\rm {\SSQMn}}
} 
\end{equation}
in contrast to \eqref{eqn:Mnat-thm-equiv}.
 To be more specific, it is easy to see that
\begin{quote}
 $\bullet$ (SSQ\Mnat-EXC-PRJ) (i) and (ii) together imply {\SSQMn},\\
 $\bullet$ {\SSQMn} implies (SSQ\Mnat-EXC-PRJ) (i),
\end{quote}
but {\SSQMn} does not imply (ii) and (iii) of (SSQ\Mnat-EXC-PRJ),
as shown in the examples below.

\begin{example}\rm
\label{ex:mnat-prj-ii-cex}
 Consider the function $f: \Z^2 \to \Rinf$ given by 
\begin{align*}
&  \dom f= \{x \in \Z^2 \mid x(1)\ge 0, \ x(2) \ge 0,\ x(1)+x(2) \le 2\},\\
& f(0,0)=0,\ f(1,0)=f(0,1)=1,\ f(0,2)=2,\ f(2,0)=f(1,1)=3
\end{align*}
(see Figure \ref{fig:mnat-prj-ii-cex}).
 Function $f$ satisfies the condition {\SSQMn}.
 The condition 
(SSQ\Mnat-EXC-PRJ) (ii) fails for this function. 
 \begin{figure}[t]
 \begin{center}
 \includegraphics[width=0.35\textwidth]{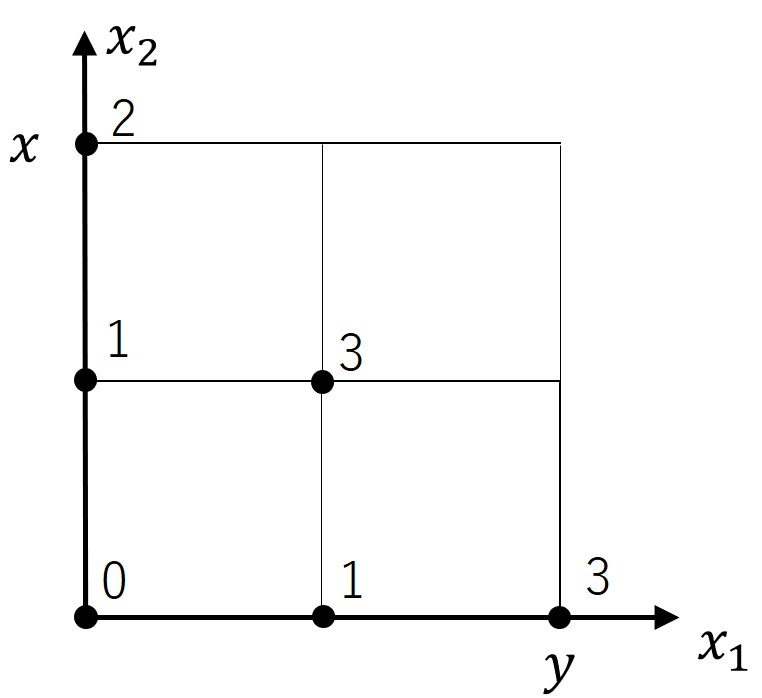}
 \caption{Values of function $f$
 in Example \ref{ex:mnat-prj-ii-cex}.} 
 \label{fig:mnat-prj-ii-cex}
 \end{center}
 \end{figure}
 Indeed, for $x=(0,2)$, $y=(2,0)$, and $i=2 \in \suppp(x-y)$,
we have a unique element $j=1 \in \suppm(x-y)$, for which
\begin{align*}
&  x - \chi_i + \chi_j =y + \chi_i - \chi_j = (1,1),\\
&  f(x) = 2 < 3 = f(x - \chi_i + \chi_j),\qquad
 f(y) =  3 = f(y + \chi_i - \chi_j).
\end{align*}
\qed
\end{example}

\begin{example}\rm
 Consider the function $f: \Z^2 \to \Rinf$ in
Example \ref{ex:mincut-counter-ex}, which satisfies {\SSQMn}.
 The condition (iii) of (SSQ\Mnat-EXC-PRJ)
fails for this function.
 Indeed, for $x= (0,1)$ and $y = (2,0)$,
we have $x(N)=1 < 2 = y(N)$ and $\suppm(x-y)=\{1\}$,
but 
\begin{align*}
 & x + \chi_1 = (1,1), \quad y- \chi_1 = (1,0), \\
 & f(x + \chi_1) = 3 > 2 = f(x), \quad
  f(y - \chi_1) = 1 > 0 = f(y).
\end{align*}
 We can also verify that
$\tilde{f}$ in  \eqref{eqn:def-tildef}
does not satisfy (SSQM), which is consistent with
\eqref{eqn:SSQMn-equiv}.
\qed 
\end{example} 

It is known that an s.s.\;quasi M-convex function
(i.e., function $f$ satisfying (SSQM))
satisfies the minimizer cut property \cite[Theorem 4.3]{MS03quasi}
and the proximity property \cite[Theorem 4.4]{MS03quasi};
it can be shown that an s.s.\;quasi M-convex function
also enjoys the geodesic property (see Section \ref{sec:geodesic-quasiM-proof} 
in Appendix).
 It follows from this fact that
if a function $f: \Z^n \to \Rinf$ satisfies
the (stronger) condition (SSQ\Mnat-EXC-PRJ), then
it satisfies the same statements
as in Theorem \ref{thm:min-cut-Mnat} (minimizer cut property),
Theorem \ref{thm:strong-min-cut-Mnat} (geodesic property),
and Theorem \ref{thm:proxmity-Mnat} (proximity property)
for \Mnat-convex functions.
 In this connection we emphasize that
the s.s.\;quasi \Mnat-convex functions in Examples \ref{ex:min-cut-qMnat-1}
and \ref{ex:proximity} do not satisfy (SSQ\Mnat-EXC-PRJ) (iii).




\appendix


\section{Appendix}

\subsection{Properties on Minimization of M-convex Functions}
\label{sec:prop-M-conv}

For ease of comparison between M- and \Mnat-convexity, we recall 
 four theorems 
for minimization of an M-convex function,
which correspond, respectively, 
 to Theorems \ref{thm:Mminimizer-Mnat},
\ref{thm:min-cut-Mnat}, \ref{thm:strong-min-cut-Mnat}, and \ref{thm:proxmity-Mnat}
for an \Mnat-convex function.

\paragraph{Optimality Condition by Local Optimality}

\begin{theorem}[{{\cite[Theorem 2.4]{Mstein96}, \cite[Theorem 2.2]{Msbmfl99}}}]
\label{thm:Mminimizer-M}
 Let $f: \Z^n \to \Rinf$ be an M-convex function.
 A vector $x^* \in \dom f$ 
is a minimizer of $f$ if and only if
\begin{align*}
f(x^* - \chi_i + \chi_j) \ge f(x^*) \qquad (i, j \in N).
\end{align*}
\end{theorem}

\noindent
The same statement holds also for s.s.\;quasi M-convex functions
\cite[Theorem 4.2]{MS03quasi}.

\paragraph{Minimizer Cut Property}

\begin{theorem}[{{\cite[Theorem 2.2]{Shi98min}}}]
\label{thm:min-cut-M}
 Let $f: \Z^n \to \Rinf$ be an M-convex function with $\arg\min f \ne \emptyset$,
and $x \in \dom f$ be a vector with $x \not\in \arg\min f$.
\\
{\rm (i)}
 For $i \in N$,
let $j \in N$ be an element 
minimizing the value $f(x - \chi_i + \chi_j)$.
 Then, there exists some minimizer $x^*$ of $f$ satisfying
\begin{equation*}
\begin{cases}
x^*(j) \ge x(j)+1 
& (\mbox{if }j \in N \setminus \{i\}),\\
x^*(i) \ge x(i) 
& (\mbox{if }j =i).
 \end{cases}
\end{equation*}
{\rm (ii)}
 Symmetrically, for $j \in N$, let $i \in N$ be an element 
minimizing the value $f(x - \chi_i + \chi_j)$.
 Then, there exists some minimizer $x^*$ of $f$ satisfying
\begin{equation*}
\begin{cases}
x^*(i) \le x(i)-1 
& (\mbox{if }i \in N \setminus \{j\}),\\
x^*(j) \le x(j) 
& (\mbox{if }i =j).
 \end{cases}
\end{equation*}
{\rm (iii)}
 For a pair $(i,j)$ of distinct elements in $N$
minimizing the value $f(x - \chi_i + \chi_j)$,
there exists some minimizer $x^*$ of~$f$ satisfying
$x^*(i) \le x(i)-1$ and
$x^*(j) \ge x(j)+1$. 
\end{theorem}

\noindent
The same statement holds also for s.s.\;quasi M-convex functions
\cite[Theorem 4.3]{MS03quasi}.

\paragraph{Geodesic Property}

Recall the definitions of $\mu(x)$ and $M(x)$ in
\eqref{eqn:def-mu} and \eqref{eqn:def-Mx}, respectively.

\begin{theorem}[{\cite[Corollary~4.2]{Shi22L1}, \cite[Theorem~2.4]{MS21steepM}}]
\label{thm:geodesic-M}
 Let $f: \Z^n \to \Rinf$ be an M-convex function with $\arg\min f \ne \emptyset$,
and $x \in \dom f$ be a vector that is not a minimizer of $f$.
 Also, let $(i,j)$ be a pair of distinct elements in $N$
 minimizing the value $f(x - \chi_i + \chi_j)$,
 and define 
 \begin{equation*}
 {M}'=
 \{x^* \in M(x)
 \mid x^*(i) \le x(i)-1,\ x^*(j) \ge x(j)+1  \}.
 \end{equation*}
 {\rm (i)}
 There exists some $x^* \in M(x)$ that is contained in ${M}'$;
 we have ${M}' \ne \emptyset$, in particular.
 \\
 {\rm (ii)}
 It holds that
$\mu(x - \chi_i + \chi_j) = 
 \mu(x)-2$
and $M(x - \chi_i + \chi_j) ={M}'$.
\end{theorem}

\noindent
The same statement holds also for s.s.\;quasi M-convex functions;
see Section \ref{sec:geodesic-quasiM-proof} for a proof.

\paragraph{Proximity Property}

\begin{theorem}[{\cite[Theorem 3.4]{MMS02}, \cite[Theorem 4.4]{MS03quasi}}]
\label{thm:proxmity-M}
 Let $f: \Z^n \to \Rinf$ be an M-convex function 
and $\alpha\ge 2$ be an integer.
 For every vector ${x} \in \dom f$  satisfying
\begin{equation*}
  f({x})  \le  \min_{i,j \in N}f({x} - \alpha (\chi_i -  \chi_j)),
\end{equation*} 
 there exists some minimizer $x^*$ of $f$ such that
$\|x^* - {x} \|_\infty \le (n-1) (\alpha -1)$.
\end{theorem}

\noindent
The same statement holds also for s.s.\;quasi M-convex functions
\cite[Theorem 4.4]{MS03quasi}.

\subsection{Proof of Geodesic Property
for Quasi M-convex Functions}
\label{sec:geodesic-quasiM-proof}

 We give a proof for the following geodesic property
for semi-strictly quasi M-convex functions.

\begin{theorem}
\label{thm:geodesic-quasiM}
 Let $f: \Z^n \to \Rinf$ be a function with $\arg\min f \ne \emptyset$
satisfying {\SSQM} (i.e., $f$ is semi-strictly quasi M-convex), 
and $x \in \dom f$ be a vector that is not a minimizer of $f$.
 Also, let $(i,j)$ be a pair of distinct elements in $N$
 minimizing the value $f(x - \chi_i + \chi_j)$,
 and define 
 \begin{equation*}
 {M}'=
 \{x^* \in M(x)
 \mid x^*(i) \le x(i)-1,\ x^*(j) \ge x(j)+1  \}.
 \end{equation*}
 {\rm (i)}
 There exists some $x^* \in M(x)$ that is contained in ${M}'$;
 we have ${M}' \ne \emptyset$, in particular.
 \\
 {\rm (ii)}
 It holds that
$\mu(x - \chi_i + \chi_j) = 
 \mu(x)-2$
and $M(x - \chi_i + \chi_j) ={M}'$.
\end{theorem}

 The proof given below is essentially the  same as the one
in \cite{MS21steepM} 
for Theorem \ref{thm:geodesic-M} on M-convex functions.
 In the proof of Theorem \ref{thm:geodesic-quasiM} (i)
we use the following lemma.

\begin{lemma}
\label{le:mc_add_near_opt}
 Let $f: \Z^n \to \Rinf$ be a function with $\arg\min f \ne \emptyset$
satisfying {\SSQM}.
 Assume that  $x \in \dom f$ is not a minimizer of $f$,
and let $(i, j)$ be a pair of distinct elements in $N$
minimizing the value $f(x - \chi_i + \chi_j)$.
 Define $y = x - \chi_i + \chi_j$.
\\
{\rm (i)} For every $x^* \in \arg\min f$ with $i \in \suppp(x^* - y)$,
there exists some $h \in \suppm(x^* - y)$ with $h \ne j$
such that
$x^* - \chi_i + \chi_h \in \arg\min f$.\\
{\rm (ii)} For every $x^* \in \arg\min f$ with $j \in \suppm(x^* - y)$,
there exists some $k \in \suppp(x^* - y)$ with $k \ne i$
such that
$x^* + \chi_j - \chi_k \in \arg\min f$.
\end{lemma}
\begin{proof}
We prove (i) only since (ii) can be proven similarly.
 We first note that $f(y) < f(x)$  holds
since $x$ is not a minimizer \cite[Theorem 4.2]{MS03quasi}.
 By definition, $f$ satisfies the condition (SSQM),
which, applied to $x^*, y$ and $i \in \suppp(x^* - y)$,
implies that there exists some $h \in \suppm(x^* - y)$ such that
at least one of the following conditions holds:
\begin{align}
& \label{eq:mc_add_near_opt_2-1}
f(x^* - \chi_i + \chi_h)<  f(x^*),
\\
& \label{eq:mc_add_near_opt_2-2}
f(y + \chi_i - \chi_h) < f(y),
\\
& \label{eq:mc_add_near_opt_2-3}
f(x^* - \chi_i + \chi_h)=  f(x^*)
\mbox{ and }
f(y + \chi_i - \chi_h) = f(y).
\end{align}
 Note that $h \neq i$ holds.
 We have 
\begin{equation}
\label{eq:mc_add_near_opt_3}
 f(x^* - \chi_i + \chi_h) \ge     f(x^*)
\end{equation}
since $x^* \in \arg \min f$.
 By the choice of $i, j \in N$, we have
\begin{equation}
\label{eq:mc_add_near_opt_4}
f(y + \chi_i - \chi_h) \ge    f(y), 
\end{equation}
where $y + \chi_i - \chi_h = x - \chi_h + \chi_j$.
 The inequalities  \eqref{eq:mc_add_near_opt_3} and \eqref{eq:mc_add_near_opt_4} 
exclude the possibility of
\eqref{eq:mc_add_near_opt_2-1} and \eqref{eq:mc_add_near_opt_2-2},
respectively.
 Therefore, we have \eqref{eq:mc_add_near_opt_2-3}.
 The former equation in \eqref{eq:mc_add_near_opt_2-3}
implies that $x^* - \chi_i + \chi_h$
is also a minimizer of $f$.
 We also have $h \ne j$,
which follows from the latter equation in \eqref{eq:mc_add_near_opt_2-3}
and the inequality $f(y) < f(x)$
since $y + \chi_i - \chi_h=x - \chi_h + \chi_j$.
\end{proof}

\begin{proof}[Proof of Theorem \ref{thm:geodesic-quasiM} (i)]
 Putting $y = x - \chi_i + \chi_j$,  we have $f(y)< f(x)$.
 We first show that there exists some $y^* \in M(x)$ such that $y^*(i) \leq x(i) -1$.
Assume, to the contrary, that $y^*(i) > x(i) - 1 =y(i)$ holds 
for every $y^* \in M(x)$.
Let $y^* \in M(x)$ be a vector minimizing the value $y^*(i)$ among all such vectors.
 By Lemma \ref{le:mc_add_near_opt}~(i), there exists 
some $h \in \suppm(y^* - y)$ with $h \ne j$ such that
$y^* - \chi_i + \chi_h \in \arg\min f$.
Since $y^*(h) < y(h) = x(h)$, we have
$\norm{(y^* - \chi_i + \chi_h) - x}_1 \le \norm{y^* - x}_1$.
 Since $y^* \in M(x)$ and $y^* - \chi_i + \chi_h \in \arg \min f$,
it follows that $y^* - \chi_i + \chi_h \in M(x)$.
This, however, is a contradiction to the choice of $y^*$
since $(y^* - \chi_i + \chi_h)(i) = y^*(i) - 1 < y^*(i)$.

We then show that
there exists some $x^* \in M(x)$ satisfying
both of  $x^*(i) \le x(i) -1$ and  $x^*(j) \ge x(j) + 1$.
Let $x^* \in M(x)$ be a vector with $x^*(i) \le x(i) - 1$. 
If there exists such $x^*$ with $x^*(j) \ge x(j) + 1$, then we are done.
Hence, we assume, to the contrary, that $x^*(j) < x(j) + 1 = y(j)$
for every $x^* \in M(x)$ with $x^*(i) \le x(i) - 1$, and suppose that $x^*$ maximizes the 
value $x^*(j)$ among all such $x^*$.
 By Lemma \ref{le:mc_add_near_opt} (ii), there exists 
some $k \in \suppp(x^* - y)$ with $k \ne i$
such that $x^* + \chi_j - \chi_k \in \arg\min f$.
 We show that $x^* + \chi_j - \chi_k \in M(x)$ holds.
Since $x^*(k) > y(k) = x(k)$, we have
$\norm{(x^* + \chi_j - \chi_k) - x}_1 \le \norm{x^* - x}_1$.
We also have $(x^* + \chi_j - \chi_k)(i) = x^*(i) \le x(i) - 1$ since $i \notin \{j, k\}$.
This, however, is a contradiction to the choice of $x^*$
since $(x^* + \chi_j - \chi_k)(j) = x^*(j) + 1 > x^*(j)$.
Hence, we have $x^*(j) \ge x(j) + 1$.

This concludes the proof of Theorem \ref{thm:geodesic-quasiM} (i).
\end{proof}

\begin{proof}[Proof of Theorem \ref{thm:geodesic-quasiM} (ii)]
 We first prove the equation $\mu (x - \chi_i + \chi_j)   = \mu(x) - 2$.
It holds that
\begin{align}
\forall y^* \in M(x - \chi_i + \chi_j): \quad
    \mu (x - \chi_i + \chi_j)
    & = \norm{y^* - (x - \chi_i + \chi_j)}_1 \nonumber \\
    &\ge \norm{y^* - x}_1 - \norm{\chi_i - \chi_j}_1 \nonumber \\
 &   = \norm{y^* - x}_1 - 2 \ge \mu(x) - 2,
     \label{eq:p2_2}
\end{align}
where the first inequality is by the triangle inequality.
  We also have
\begin{align}
\forall x^* \in M': \quad
      \mu (x - \chi_i + \chi_j) 
 \le \norm{x^* - (x - \chi_i + \chi_j)}_1
& =     \norm{x^* - x}_1 - 2 
= \mu(x) - 2,
     \label{eq:p2_1}
\end{align}
where the first equality is by the inequalities
$x^*(i) \leq x(i) - 1,\ x^*(j) \geq x(j) + 1$
and the second equality is by $x^* \in M(x)$.
 The equation $\mu (x - \chi_i + \chi_j)   = \mu(x) - 2$
 follows from \eqref{eq:p2_2} and \eqref{eq:p2_1}.

 We then prove the equation $M(x - \chi_i + \chi_j) ={M}'$.
 It follows  from $\mu (x - \chi_i + \chi_j)   = \mu(x) - 2$,
\eqref{eq:p2_2}, and \eqref{eq:p2_1} that 
all the inequalities in \eqref{eq:p2_2} and \eqref{eq:p2_1} hold
with equality; in particular, we have
\begin{align}
& 
\forall y^* \in M(x - \chi_i + \chi_j): \quad
\norm{y^* - (x - \chi_i + \chi_j)}_1 = \norm{y^* - x}_1 - 2 = \mu(x) - 2,
     \label{eq:claim23-2}
\\
&\forall x^* \in M': \quad
\mu (x - \chi_i + \chi_j)   =
 \norm{x^* - (x - \chi_i + \chi_j)}_1.
     \label{eq:claim23-3}
\end{align}
 The equation \eqref{eq:claim23-3} implies that
$M(x - \chi_i + \chi_j) \supseteq M'$.
 We can also obtain the reverse inclusion
$M(x - \chi_i + \chi_j) \subseteq M'$ from \eqref{eq:claim23-2}.
 Indeed, for every $y^* \in M(x - \chi_i + \chi_j)$,
we have $y^* \in M(x)$ by $\norm{y^* - x}_1 = \mu(x)$,
and we also have
 $y^*(i) \le x(i) - 1$ and $y^*(j) \ge x(j) + 1$
since $\norm{y^* - (x - \chi_i + \chi_j)}_1 = \norm{y^* - x}_1 - 2$.
 Hence, $M(x - \chi_i + \chi_j) \subseteq M'$ holds,
implying the equation $M(x - \chi_i + \chi_j) ={M}'$.
\end{proof}



\newpage

\tableofcontents

\end{document}